\documentclass[11pt]{article}
\usepackage{color} 
\usepackage{amssymb,amsmath,mathrsfs,amsthm}
\usepackage{enumitem}
\usepackage{mathtools}
\usepackage{geometry}
\usepackage{amsfonts}
\usepackage{color}

\usepackage{esint,comment}

\usepackage[dvips]{graphicx}
\allowdisplaybreaks
\newtheorem{thm}{Theorem}

\newtheorem{corollary}[thm]{Corollary} 

 \newtheorem{lemma}[thm]{Lemma}
\newtheorem{prop}[thm]{Proposition}

\theoremstyle{definition}

\def\indeq{\quad{}} 

\def\comma{ {\rm ,\qquad{}} }            
\def\commaone{ {\rm ,\quad{}} }

\def\colb{\color{black}}

\def \no#1#2#3 {{\bf #1} (#3), #2.}
\def \eds#1#2#3 {#1, #2, #3.}

\def\d{{\rm d}}

\def\Q{{\EuScript Q}}

\def\N{{\mathbb N}}

\def\:{{\colon}}

\def\be#1{\begin{equation}\label{#1}}
\def\ee{\end{equation}}

\def\<{\langle}
\def\>{\rangle}
\def\coloneqq{:=}

\newcommand{\na}{\nabla}

\newcommand{\pp}{\pi}
\newcommand{\lec}{\lesssim}
\newcommand{\bs}{\begin{split}}
\newcommand{\essss}{\end{split}}

\newcommand{\eqnb}{\begin{equation}}
\newcommand{\eqne}{\end{equation}}

\renewcommand{\ee}{\mathrm{e}}

\newcommand{\p}{\partial}

\newcommand{\RR}{\mathbb{R}}

\newcommand{\ZZ}{\mathbb{Z}}

\renewcommand{\d}{\mathrm{d}}
\renewcommand{\Q}{\mathcal{Q}}

\newcommand{\supp}{\operatorname{supp}}

\newcommand\blfootnote[1]{%
  \begingroup
  \renewcommand\thefootnote{}\footnote{#1}%
  \addtocounter{footnote}{-1}%
  \endgroup
}

\begin{document}
\title{An anisotropic regularity condition for the 3D incompressible Navier-Stokes equations for the entire exponent range}
\author{I. Kukavica, W. S. O\.za\'nski}
\date{\vspace{-5ex}}
\maketitle
\blfootnote{I.~Kukavica: Department of Mathematics, University of Southern California, Los Angeles, CA 90089, USA, email: kukavica@usc.edu\\
W.~S.~O\.za\'nski: Department of Mathematics, University of Southern California, Los Angeles, CA 90089, USA, email: ozanski@usc.edu\\ 
I.~Kukavica was supported in part by the
NSF grant DMS-1907992. W.~S.~O\.za\'nski was supported in part by the Simons Foundation. }

\begin{abstract}
We show that a suitable weak solution to the incompressible Navier-Stokes equations on ${\RR^3\times(-1,1)}$ is regular on $\RR^3\times(0,1]$ if $\p_3 u $ belongs to $M^{2p/(2p-3),\alpha } ((-1,0);L^p (\RR^3 ))$ for any $\alpha >1$ and $p\in (3/2,\infty)$, which is a logarithmic-type variation of a Morrey space in time. For each $\alpha >1$ this space is, up to a logarithm, critical with respect to the scaling of the equations, and contains all spaces  $L^q ((-1,0);L^p (\RR^3 ))$ that are subcritical, that is for which $2/q+3/p<2$.
\end{abstract}


\section{Introduction}
We address conditional regularity of suitable Leray-Hopf weak solutions to the incompressible Navier-Stokes equations (NSE),
\eqnb\label{NSE_intro}
\begin{split}
u_t - \Delta u + u\cdot \na u + \na \pp &=0 ,\\
\mathrm{div}\, u &=0,
\end{split}
\eqne
in ${\mathbb R}^{3}\times(0,T)$. Our main result is the following.

\begin{thm}\label{thm_main}
Suppose that $(u, \pp)$ is a suitable Leray-Hopf weak solution to the Navier-Stokes
equations on $\RR^3\times (-1,1)$ such that for some $\alpha>1$ and
$p\in (3/2,\infty )$ we have
\eqnb\label{def_of_morrey}
\| \p_3 u \|_{L^\frac{2p}{2p-3} (I;L^p )} \leq C_{p,\alpha } \left(\frac{-1}{\log |I|} \right)^{\alpha}
\eqne
for every $I\subset (-1,0)$ with $|I|<\frac12$. Then $u$ is regular on $\RR^3\times (-1,0]$. 
\end{thm}

Here we write $L^p \equiv L^p (\RR^3)$, for brevity. 

In order to put this result in a context, we note that the study of conditional regularity of the NSE goes back to Serrin, Ladyzhenskaya, and
Prodi (\cite{serrin,ladyzhenskaya,prodi}), who proved that if
$u\in L_t^{q}L_x^{p}$ holds with $2/q+3/p\leq1$, where $p\in(3,\infty]$ then the solution is regular.
On the other hand, Beir\~ao da~Veiga showed in~\cite{B} that the regularity holds if
$\nabla u\in L_t^{q}L_x^{p}$  with $2/q+3/p\leq2$ and $p\in(3/2,\infty)$.

In \cite{NP1}, Neustupa and Penel proved that boundedness of only one component of the velocity (say $u_3$) implies regularity, 
with the approach based on the evolution equation for $\omega_3$ (cf.~also \cite{NP2}).
 Afterwards, there have been many results \cite{CC,H,NNP,P,PP1,SK}, which approached the Serrin's scale invariant condition in terms of one velocity component, until a recent breakthrough paper~\cite{chae_wolf}, which  achieved the range of exponents with strict inequality $2/q+3/p<1$. A subsequent paper \cite{wang10} has improved it up to the equality, but with the Lorenz spaces replacing Lebesgue spaces for integrability in time.

As for regularity conditions in terms of $\p_3u$, 
Penel and Pokorn\'y proved in \cite{PP1} regularity under the condition
that $\partial_{3}u$ belongs to 
$ L_t^{q}L_x^{p}$ where
$2/q+3/p\leq 3/2$ and $2\leq p\leq \infty$.
The  result in \cite{KZ} then provided a  scale invariant regularity
criterion
$2/q+3/p\leq 2$,
with a restricted exponent range $9/4\leq p\leq3$. 
The method in \cite{KZ} was based on testing the equations for $(u_1,u_2)$ with $-\Delta_2 u_{1,2}$,
and an identity for $\sum_{i,j=1}^2\int u_i\partial_{i}u_j\Delta u_j$ in which every
term contains $\partial_{3}u$.
The partial regularity methods 
\cite{CKN,V,ozanskiCKN,W1}
allowed localization of this condition in \cite{KRZ_localized}.
There have been several improvements on the criteria since then;
cf.~\cite{BG,CZ,PP2,Sk1,Sk2} for a partial list of references.
In particular, in \cite{Sk2}, Skal\'ak extended the range 
for $\partial_{3}u$
to $3/2<p\leq 3$ using sharp anisotropic inequalities, and, very recently, this range has been extended to $3/2<p\leq 6$ in \cite{chen_fang_zhang}.

In this context Theorem~\ref{thm_main} provides the first conditional regularity criterion in terms of $\partial_{3} u$ covering the full range of Lebesgue exponents $3/2<p<\infty$ as well as all Lebesgue spaces $L^q_t L^p_x$ with sharp inequality $2/q+2/p<2$. To be more precise, letting \eqref{def_of_morrey} be the definition of a Morrey-type space $M^{\frac{2p}{2p-3},\alpha} ((-1 ,0 ); L^p )$, we immediately see that such Morrey space contains $L^q ((-1,0);L^p )$ for every $q>\frac{2p}{2p-3}$ (that is such that $2/q+3/p<2$), since H\"older's inequality implies
\[
\| \p_3 u \|_{L^\frac{2p}{2p-3} (I;L^p )} \leq \| \p_3 u \|_{L^q (I;L^p )} |I|^{\frac{2pq-3q-2p}{2pq}} \leq C_{p,q,\alpha } \left(\frac{-1}{\log |I|} \right)^{\alpha}
\]
for any $I\subset (-1,0)$, $|I|\leq \frac12$. Furthermore, $M^{\frac{2p}{2p-3},\alpha} ((-1 ,0 ); L^p )$ is, up to a logarithm, critical with respect to the scaling of the equations; namely letting $u_\lambda \coloneqq \lambda u (\lambda x, \lambda^2 t)$ we have
\[
\| \p_3 u_\lambda  \|_{L^{\frac{2p}{2p-3}} (I ; L^p ) } = \| \p_3 u  \|_{L^{\frac{2p}{2p-3}} (\lambda I ; L^p ) } \leq C_{p,\alpha } \left(\frac{-1}{\log |\lambda I|} \right)^{\alpha} \leq C_{p,\alpha } \left(\frac{-1}{\log |I|} \right)^{\alpha} O\left( \left( \frac{-1}{\log \lambda }\right)^\alpha \right)  
\]
as $\lambda \to 0^+$, for every $I\subset (-1,0)$ such that $|I|\leq \frac12$. 

The question whether the Morrey space $M^{\frac{2p}{2p-3},\alpha} ((-1 ,0 ); L^p )$ can be replaced by a critical Lebesgue-type space, $L^{\frac{2p}{2p-3}}((-1 ,0 ); L^p )$, without any restriction on the range of $p$ as in Theorem~\ref{thm_main}, remains an open problem.

Our approach in proving Theorem~\ref{thm_main} is inspired by the treatment of a related regularity condition in terms of one component of $u$ that was recently proved by Wang et~al in \cite{wang10}, which in turn drew from a recent result of Chae and Wolf~\cite{chae_wolf}, which introduced a new approach
based on partial regularity and testing the local energy equality with a one-dimensional backward heat kernel.

\section{Proof of Theorem~\ref{thm_main}}
Before proceeding to the proof of our main result, we recall that $(u,\pp)$ with
\eqnb\label{def_of_E_finite}
E \coloneqq \| u \|^2_{L^\infty ((-1,0);L^2)} + \| \na u \|^2_{L^2 (\RR^3\times(-1,1))}  <\infty 
\eqne
is a suitable weak solution in $\RR^{3}\times(-1,1)$ if it satisfies the equation~\eqref{NSE_intro} in the distributional sense, if $\pp = (-\Delta )^{-1} (\p_i u_j \p_j u_i)$, the strong energy inequality
\[
\int |u(t)|^2 + 2 \int_s^t |\nabla u |^2 \leq \int |u(s)|^2
\]
holds for almost all $s\in (-1,1)$ and all $t\in (s,1)$, as well as the local energy inequality
\[
\int_{\RR^3} |u(t) |^2 \phi(t) + 2  \int_{-1}^t \int_{\RR^3} | \nabla u |^2\phi \leq \int_{-1}^t \int_{\RR^3 } \left( |u|^2 (\p_t \phi + \Delta \phi )  +(|u|^2+2\pp)(u\cdot \nabla )\phi   \right)\]
holds for all $t\in (-1,1)$ and $\phi \in C_0^\infty (\RR^3\times(-1,1) ; [0,1])$. We note that, due to the global integrability assumptions on $u$, the local energy inequality can be extended to include the test functions $\phi \in C^\infty (\RR^3\times(-1,1) ; [0,1])$ that have compact support only in time and that have bounded derivatives. For $n\in \N_0$ we set $r_n\coloneqq 2^{-n}$ and 
\[\begin{split}
U_n & \coloneqq \RR^2 \times (-r_n , r_n)
\comma
Q_n \coloneqq U_n \times (-r_n^2 ,0 ).
\end{split}
\]
We also set 
\[\begin{split}
E_n &\equiv E(r_n )\coloneqq\sup_{t\in (-r_n^2,0)} \int_{U_n} | u(t) |^2 \d x + \int_{-r_n^2}^0 \int_{U_n} |\na u |^2 \d x \,\d s
\end{split}\]
and
\[\begin{split}
\Phi_n (x_3,t) &\coloneqq (4 \pi (r_n^2-t))^{-\frac12} \ee^{-\frac{x_3^2}{4(r_n^2-t)}}
   \comma x_3\in \RR
   \commaone t<0.
\end{split}\]
Note that
\eqnb\label{Phi_n_pointwise}
r_k^{-1} \lec \Phi_n \lec r_k^{-1} \quad \text{ on } Q_k
\comma k= 0,1,\ldots , n.
\eqne
Fix $p\in (3/2, \infty) $ and set
\[
B_k \coloneqq \|\p_3 u \|_{L^{\frac{2p}{2p-3}}_t L^p_x (Q_k)}.
\]
Note that 
\eqnb\label{bk_sum}
\sum_{k\geq 0} B_k \lec_{p,\alpha   } \sum_{k\geq 0 } \left(\frac{-1}{\log | r_k^2|} \right)^{\alpha} \lec_\alpha \sum_{k\geq 0 } k^{-\alpha } \lec_\alpha  1  ,
\eqne
by the assumption \eqref{def_of_morrey}. We also set $B\coloneqq \|\p_3 u \|_{L^{\frac{2p}{2p-3}} ((-1,0);L^p )}$.

In order to prove the main result, Theorem~\ref{thm_main}, we need the following localization property.

\begin{prop}\label{prop_main}
Let $(u, \pp)$ be a suitable weak solution to the NSE on $\RR^3\times (-1,1)$ satisfying \eqref{def_of_morrey}. Then
\[
\frac{1}{r_n} E_n \lec E (1+B) +  \sum_{j,k=0 }^{n-1}( r_k^{-1} E_k )^{\frac12} ( r_j^{-1} E_j )^{\frac12} a_{jk }
\]
for all $n\geq 0$, where
\[
a_{jk} \coloneqq  \chi_{k\geq j } r_k^{\frac2a } r_j^{-\frac2a} B_j +   \delta_{jk} r_k^{\frac2a}B+\delta_{jk} B_k
\]
for some $a>1$.
\end{prop}
Here $\chi_{k\geq j} = 1$ for $k\geq j$ and $0$ otherwise, and $\delta_{kj}$ denotes the Kronecker delta.
\begin{proof}[Proof of Proposition~\ref{prop_main}] Let $\eta (x_3,t)$ be such that $\supp \eta \Subset (-\frac12,\frac12 ) \times (-\frac12,0]$ and $\eta=1$ on $(-\frac14,\frac14 ) \times (-\frac1{16} , 0)$.  The local energy inequality applied with $\Phi_n \eta$, where $n\in \N$ is fixed, gives
\eqnb\label{LEI_pf}\begin{split}
 \int_{U_0}& |u(t)|^2 \Phi_n(t) \eta(t) \d x + 2\int_{-1}^t\int_{U_0} |\na u |^2 \Phi_n \eta \\
& \lec  \int_{-1}^t \int_{U_0} |u|^2 (\p_t +\Delta )(\Phi_n \eta ) + \int_{-1}^t \int_{U_0} (|u|^2 +2\pp )u\cdot \na (\Phi_n \eta )
 \end{split}
\eqne
for almost every $t\in (-1,0)$. We show below that the right-hand side can be bounded from above by a constant multiple of $E_0 +  \sum_{k=0 }^n r_k^{-1} E_k (r_k^{\varepsilon}B_0 + B_k )$, uniformly in $t$. This and the bound $\Phi_n \eta \gtrsim r_n^{-1}$ on $Q_n$ then give the claim.

For the first term on the right-hand side of \eqref{LEI_pf}, we have
\[
\begin{split}
\int_{Q_0} |u|^2 \left| (\p_t +\Delta )(\Phi_n \eta )\right|&= \int_{Q_0} |u|^2 |2 \p_3 \Phi_n  \p_3 \eta + \Phi_n \p_{33} \eta |  
\lec \int_{Q_0} |u|^2 \lec E_0,
\end{split}
\]
where we used that $\Phi_n$ satisfies the one-dimensional heat equation in~$Q_0$ in the first step, and the bounds $|\na \Phi_n |, | \Phi_n | \lec 1$ on $(\supp\, \p_3 \eta) \cap Q_0$ in the second step. 

For the velocity component of the second term on the right-hand side of \eqref{LEI_pf}, we have
\[
\begin{split}
& \int_{-1}^{t}\int_{U_0} |u|^2 u\cdot \na (\Phi_n \eta ) 
  = -2\int_{-1}^t \int_{U_0} \p_3 u \cdot u\, u_3 \Phi_n \eta 
        -  \int_{-1}^{t}\int_{U_0} |u|^2\partial_{3} u_3 \Phi_n\eta
  \\&\indeq
  \lec \sum_{k=0}^{n-1} \int_{Q_k\setminus Q_{k+1}} |\p_3 u |\, |u |^2 \Phi_n \eta  +  \int_{Q_n} |\p_3 u |\, |u |^2 \Phi_n \eta 
  \lec \sum_{k=0}^{n-1} r_k^{-1} \int_{Q_k} |\p_3 u |\, |u |^2  +  r_n^{-1} \int_{Q_n} |\p_3 u |\, |u |^2  
  \\&\indeq
  \lec \sum_{k=0}^{n} r_k^{-1} \| u \|^2_{L^{\frac{4p}3}_t L^{2p'}_x(Q_k)} \|\p_3 u \|_{L^{\frac{2p}{2p-3}}_t L^p_x (Q_k)}  \lec \sum_{k=0}^{n} r_k^{-1} E_k  B_k  
   \lec \sum_{k=0}^{n-1} r_k^{-1} E_k  B_k  
  ,
\end{split}
\]
where we have estimated the last term, $k=n$, using the term with $k=n-1$, in the last step.

As for the term in \eqref{LEI_pf} involving the pressure
$
\pp=(-\Delta )^{-1} (\p_i u_j \p_j u_i)
$, for each $k\in \N_0$
we choose
$\chi_k (x_3,t) \in 
C_0^\infty ((-r_k, r_{k})\times(-r_k^2,0];[0,1])$ such that $\chi_k =1$ on $(-r_{k+1}, r_{k+1}) \times (-r_{k+1}^2,0]$ and set
\[
\phi_j \coloneqq \begin{cases} \chi_j-\chi_{j+1}, \qquad &j=0,\ldots , n-1,\\
\chi_n, &j=n.
\end{cases}
\]
Then we may write
\eqnb\label{p_splitting}\begin{split}
\frac12 \p_3 \pp &=  (-\Delta )^{-1} \p_i \p_l ( u_l \p_3 u_i)=(-\Delta )^{-1} \p_i \p_l ( u_l \p_3 u_i \chi_0 )+(-\Delta )^{-1} \p_i \p_l ( u_l \p_3 u_i (1-\chi_0) ) \\
&=\sum_{j=0}^n(-\Delta )^{-1} \p_i \p_l ( u_l \p_3 u_i \phi_j )+(-\Delta )^{-1} \p_i \p_l ( u_l \p_3 u_i (1-\chi_0) ) =: \sum_{j=0}^n p_j + q
,
\end{split} \eqne
and thus the pressure term may be decomposed as
\eqnb\label{pressure_term_parts}
- \frac12 \int_{-1}^t \int_{U_0} \pp \,u\cdot \na (\Phi_n \eta )
  = \frac12  \int_{-1}^t \int_{U_0} \p_3 \pp \,u_3\, \Phi_n \eta 
+\frac12\int_{-1}^t \int_{U_0} \pp\partial_{3}u_3 \Phi_n \eta 
.
\eqne
Using the notation in \eqref{p_splitting}, we rewrite the first term as
\[
\begin{split} 
&\frac12  \int_{-1}^t \int_{U_0} \p_3 \pp \,u_3\, \Phi_n \eta = \sum_{j=0}^n \int_{-1}^t \int_{U_0}   p_j  \,u_3\, \Phi_n \eta  + \int_{-1}^t \int_{U_0}   q  \,u_3\, \Phi_n \eta  
\\&\indeq
 =  \sum_{k=0}^n \sum_{j= \max \{0,k-3\} }^n \int_{Q_{k}}   p_j  \,u_3\,\Phi_n \phi_k \eta  +  \sum_{j=0}^{n-4} \sum_{k= j+4 }^n \int_{Q_{k}}   p_j  \,u_3\, \Phi_n \phi_k \eta     + \int_{-1}^t \int_{U_0}   q  \,u_3\, \Phi_n \eta 
\\&\indeq
=: I_1 + I_2 + I_3 .
\end{split}
\]
For $I_1$, we note that $\sum_{j=k-3}^{n} p_j = (-\Delta )^{-1} \p_i \p_l (u_l \p_3 u_i \chi_{k-3} )$ for $k \geq 3$, which gives
\[\begin{split}
|I_1|&\lec \sum_{k=0}^n r_k^{-1} \int_{Q_{k}}   \biggl| \sum_{j=\max \{0, k-3 \}}^{n} p_j  \,u_3 \biggr|\lec \sum_{k=0}^n r_k^{-1}\| u \otimes \p_3 u \|_{L_t^{\frac{4p}{4p-3}} L^{\frac{2p}{p+1}}_x (Q_{k})} \| u \|_{L_t^{\frac{4p}{3}}L_x^{2p'}(Q_{k})}  \\
&\lec \sum_{k=0}^n r_k^{-1}\|  \p_3 u \|_{L_t^{\frac{2p}{2p-3}} L^{p}_x (Q_{k})} \| u \|^2_{L_t^{\frac{4p}{3}}L_x^{2p'}(Q_{k})}  \lec  \sum_{k=0}^{n} r_k^{-1} E_k B_k   \lec  \sum_{k=0}^{n-1} r_k^{-1} E_k B_k  ,
\end{split}
\]
as required.
For $I_2$, we note that $p_j$ is harmonic with respect to the spatial variables in $Q_{j+2}$, and thus using the anisotropic interior estimates for harmonic functions (cf.~\cite[Lemma~A.2]{chae_wolf})
we obtain
\eqnb\label{harm_est}
\| p_j \|_{L^m (\RR^2 \times (-r_{k},r_{k}))} \lec r_{k}^{\frac{1}{m}}r_j^{\frac{2}{m}-\frac3{l}}\| p_j \|_{L^l (\RR^2 \times (-r_{j+2},r_{j+2}))} 
\eqne
for all $l\in [1,m ]$. We fix any $a>\max \{ 2p/3, 4/3 \}$ and then fix any $l \in (\max \{1, 2p/(p+2) \}, \min \{ 6p/(p+6), 6a/(3a+4) \})$, which is nonempty due to our choice of $a$, to obtain 
\[
\begin{split}
|I_2| &\lec \sum_{j=0}^{n-4} \sum_{k = j+4 }^{n} r_k^{-1} \int_{Q_{k}}  | p_j  \,u | \\
&\lec \sum_{j=0}^{n-4} \sum_{k = j+4 }^{n} r_k^{-1} \| p_j \|_{L^{a'}_t L^{2}_x (Q_k)} \| u \|_{L^a_t L^2_x (Q_k)} \\
&\lec \sum_{j=0}^{n-4} \sum_{k = j+4 }^{n} r_k^{-\frac{1}{2}} r_j^{1 - \frac3l } \|u \|_{L^{\frac{2pa }{3a - 2p}}_t L_x^{\frac{lp}{p-l}} (Q_j)} \| \p_3 u \|_{L^{\frac{2p}{2p-3}}_t L^p_x (Q_j)}  \| u \|_{L^a_t L^2_x (Q_k)} ,
\end{split}
\]
where we used the harmonic estimate \eqref{harm_est}  (note that $l\leq 2 $, as required by \eqref{harm_est}) and the fact that $l<p$ in the third inequality. Note also that our choice of $l$ gives that $lp/(p-l) \in (2,6)$, which will allow us to estimate the term with the $L_x^{{lp}/(p-l)}$ norm using the energy $E$. We obtain
\[\begin{split}
|I_2| &\lec \sum_{j=0}^{n-4} \sum_{k = j+4 }^{n} r_k^{-\frac{1}{2}} r_j^{1 - \frac3l } \|u \|_{L^{\frac{4lp}{3lp -6p+6l}}_t L_x^{\frac{lp}{p-l}} (Q_j)} r_j^{\frac{3}{l}-\frac2a -\frac32} \| \p_3 u \|_{L^{\frac{2p}{2p-3}}_t L_x^p (Q_j)}  \| u \|_{L^{\infty }_t L_x^2 (Q_k)} r_k^{\frac2a  }\\
&\lec\sum_{j=0}^{n-4} \sum_{k = j+4 }^{n}  \left( r_j^{-1} E_j\right)^{\frac12}\left( r_k^{-1} E_k\right)^{\frac12}  \| \p_3 u \|_{L^{\frac{2p}{2p-3}}_t L_x^p (Q_j)}  r_k^{\frac2a } r_j^{ - \frac2a  } \\
&\lec \sum_{j=0}^{n-1} \sum_{k = j+4 }^{n-1}  \left( r_j^{-1} E_j\right)^{\frac12}\left( r_k^{-1} E_k\right)^{\frac12}  r_k^{\frac2a } r_j^{ - \frac2a  } B_j ,
\end{split}
\]
as required, where we used H\"older's inequality in time in the first inequality (hence the upper bound $l< 6a/(3a+4)$ in our choice of $l$) in order to be able to bound all norms of $u$ by $E_k^{1/2}$ or $E_j^{1/2}$ in the second inequality (where we also moved $r_k$'s and $r_j$'s around), and we have estimated the case of $k=n$ in terms of the case $k=n-1$ as well as used \eqref{def_of_morrey} in the last inequality.

The estimate on $I_3$ is analogous, but does not require summation in~$j$. Indeed, recalling \eqref{p_splitting} we see that $q$ is harmonic in $(\supp\, \eta ) \cap (U_0\times (-1,t))$, and so we perform the same estimate as in the first four inequalities in the estimate on $|I_2|$ above, but with  $Q_j$ replaced by $\RR^3 \times (-1,0)$ and without the summation in~$j$. We obtain
\[
\begin{split}
|I_3| &\lec \| u\|_{L^{\frac{4lp}{3lp-6p+6l}} \left( (-1,0);L^{\frac{lp}{p-l}} \right)}  \, \| \p_3 u \|_{L^{\frac{2p}{2p-3}} ((-1,0);L^{p})} \sum_{k= 0 }^n r_k^{-\frac12 +\frac2a }  E_{k}^{\frac12} \\
&\lec E^{\frac12} B \left( \sum_{k=0}^{n-1} r_k^{-1+\frac2a  } E_k \right)^{\frac12} \lec BE+ B \sum_{k=0}^{n-1} r_k^{-1+\frac2a } E_k,
\end{split}
\]
as required.

It remains to estimate the second term in \eqref{pressure_term_parts}. For this we apply the splitting \eqref{p_splitting} to $\pp$ (rather than to $\p_3 \pp$) to obtain
\[  \pp = \sum_{j=0}^n(-\Delta )^{-1} \p_i \p_l ( u_l  u_i \phi_j )+(-\Delta )^{-1} \p_i \p_l ( u_l  u_i (1-\chi_0) ) =: \sum_{j=0}^n \tilde p_j + \tilde q,
\]
which allows us to estimate 
\[\begin{split}
\int_{-1}^t \int_{U_0} \pp\partial_{3}u_3 \, \Phi_n \eta  &\lec \sum_{k=0}^n r_k^{-1} \int_{Q_k} | (-\Delta )^{-1} \p_i \p_m (u_iu_m \chi_{\max\{0,k-3\}} )| \,| \p_3 u | \\
&+ \sum_{j=0}^{n-4} \sum_{k=j+4}^n r_k^{-1}  \int_{Q_k} |\tilde p_j \p_3 u | +  \sum_{k=0}^n {r_k}^{-1}  \int_{Q_k} |\tilde q \,\p_3 u |  \\
&\lec  \sum_{k=0}^n  r_k^{-1} \| u \|_{L^{\frac{4p}3}_t L^{2p'}_x (Q_k)}^2 B_k \\
&+ \sum_{j=0}^{n-4} \sum_{k=j+4}^n {r_k}^{-1}  \| \tilde p_j \|_{L_t^{a'} L_x^{p'} (Q_k)} \| \p_3 u \|_{L_t^{a}L_x^p (Q_k)}  +  \sum_{k=0}^n {r_k}^{-1}  \| \tilde q \|_{L_t^{a'} L_x^{p'} (Q_k)} \| \p_3 u \|_{L_t^{a}L_x^p (Q_k)}
,
\end{split}
\]
\colb
where $a\in (1,p')$. Note such choice of $a$ implies that $a<2p/(2p-3)$. Therefore, choosing any $l\in (1, 3a/(a+2))$ we obtain 
\[\begin{split}
\int_{-1}^t \int_{U_0} \pp\partial_{3}u_3 \, \Phi_n \eta  
&\lec  \sum_{k=0}^n  r_k^{-1} \| u \|_{L^{\frac{4p}3}_t L^{2p'}_x (Q_k)}^2 B_k \\
&+ \sum_{j=0}^{n-4} \sum_{k=j+4}^n r_k^{-\frac{1}{p}}  r_j^{\frac{2}{p'}-\frac{3}l}\| \tilde p_j \|_{L_t^{a'} L_x^{l} (Q_{j+2})} r_k^{\frac{2}{a} -2+ \frac{3}p} B_k   +  \| \tilde q \|_{L_t^{a'} L_x^{l} (Q_0)}  \sum_{k=0}^n  r_k^{-\frac{1}{p}}  r_k^{\frac{2}{a} -2+ \frac{3}p} B_k \\
&\lec \sum_{k=0}^n   r_k^{-1} E_k B_k +   \sum_{j=0}^{n-4}   r_j^{\frac{2}{p'}-\frac{3}l}\| u \|^2_{L_t^{2a'} L_x^{2l} (Q_j)} B_j \sum_{k\geq j } r_k^{2\left( \frac{1}{a} -\frac1{p'}  \right) }   + EB \\
&\lec \sum_{k=0}^{n-1}   r_k^{-1} E_k B_k    + EB,
\end{split}
\]
as required, where, in the first inequality, we used the harmonic estimate \eqref{harm_est} (note that $l< p'$, as required, thanks to our choice of $a$ and $l$), as well as H\"older's inequality $\| f\|_{L^a (-r_k^2,0)} \lec \| f\|_{L^{\frac{2p}{2p-3}} (-r_k^2,0)} r_k^{\frac2a-2+\frac3p}$; in the second inequality we used the fact that $p'\in (1,3)$, which allowed us to estimate $\| u \|_{L^{\frac{4p}3}_t L^{2p'}_x (Q_k)}^2$ by $E_k$, as well as the fact that $l<p'$ to sum the last series, and the fact that $B_j$'s are nonincreasing, i.e., $B_k\leq B_j \leq B$. In the last inequality, we used H\"older's inequality $\| f \|^2_{L^{2a'}(-r_j^2 ,0)} \leq \| f \|_{L^{\frac{4l'}3}(-r_j^2 ,0)}^2  r_j^{ -1 - \frac2{a} + \frac3{l}  }$ (since $2a'< 4l'/3$ by the choice of $l$).
\end{proof}
\begin{corollary}
Under the assumptions of Proposition \ref{prop_main} we have $r^{-1} E(r)  \lec E(1+B) \mathrm{e}^{C\sum_{k,j\geq 0} a_{jk}} < \infty $ for all $r\in (0,1)$.
\end{corollary}
\begin{proof}
Recall that if $b_n, x_n\geq 0$ and $C>0$ are such that $x_0\leq C$ and $
x_n \leq C + \sum_{k=0}^{n-1} b_k\,x_k$ for all $n\geq 1$, then
\[
x_n \leq C\, \ee^{ \sum_{k\geq 0} b_k } \qquad \text{ for all }n\geq 1,
\]
by the discrete Gronwall inequality (see Lemma A.1 in \cite{wang10} for a proof). Letting $x_n \coloneqq r_n^{-1} E_n $ and $b_n \coloneqq \sum_{k\geq 0} (a_{kn}+a_{nk} )$ and using Young's inequality $ab\leq (a^2+b^2)/2$, Proposition \ref{prop_main} gives
\[
x_n \lec E(1+B) + \sum_{k,j=0}^{n-1} x_k^{\frac12} x_j^{\frac12} a_{jk} \lec E(1+B) + \sum_{k=0}^{n-1} x_k b_{k}
\]
for each $n\geq 1$. Since \eqref{bk_sum} implies that $\sum_{k\geq 0} b_k <\infty$, we obtain the claim for $r=r_n$, where $n\geq 0$. The claim for other $r$ follows by approximating with a neighboring $r_n$. 
\end{proof}
We note in passing that in the proof above we have in fact used a discrete Gronwall inequality of the form $x_n \leq C \ee^{\sum_{i,j\geq 0} a_{ij}}$ whenever $(x_n)_{n\geq 0}$ is a nonnegative sequence such that $x_0\leq C$ and $x_n \leq C + \sum_{i,j= 0}^{n-1} a_{i,j} x_i^{1/2} x_j^{1/2}$ for $n>0 $, and the coefficients $a_{ij}\geq 0$ are such that $\sum_{i,j} a_{ij} < \infty$.

We can now prove our main result.
\begin{proof}[Proof of Theorem~\ref{thm_main}]
By the above corollary, $  r^{-1} E(r)  \leq C_{E,B}$ for all $ r\in (0,1)$. Since  
\[
\| \p_3 u \|_{L^{\frac{2p}{2p-3}} ((-r^2 ,0); L^{p}(B_r))} \to 0 \quad \text{ as }\quad r\to 0, 
\]
the next lemma gives that $(0,0)$ is a regular point of $u$, in the sense that $u$ is essentially bounded in $(-\rho^2 , 0) \times B(0,\rho )$ for some $\rho $. Regularity at any other point in $\RR^3$ at $t=0$ follows analogously, by translating and rescaling $U_n$, $Q_n$, $A_n$, and $E_n$. We now show that this implies that $u$ is regular on $(-1,0]$. Indeed, note that, due to the existence of intervals of regularity of any Leray-Hopf weak solution (see Theorem 6.41 in \cite{OP}, for example) we can assume, by rescaling, that $u$ is regular on $(-1,0)$. The fact that $u$ is regular at every point $x\in \RR^3$ at $t=0$ shows that $\| u(t) \|_{L^\infty (B_R )}$ remains bounded as $t\to 0^-$ for each $R>0$. 

Due to the partial regularity theory of Caffarelli-Kohn-Nirenberg, there exists $\epsilon >0$ such that if $\int_{-1}^0 \int_{B_1(x)} \left( |u|^3 + |p|^{3/2} \right) \leq \epsilon $ then $|u| \leq C(\epsilon )$ on $(-1/4,0)\times B_{1/2} (x)$, where $C(\epsilon ) >0$ is independent of $x$ (see Theorem 2.2 in \cite{ozanskiCKN}, for example). Let $\{ B_1 (x_n ) \}_{n\geq 1 } $ be a cover of $\RR^3$ such that $x_n \in \ZZ^3/4$ for each $n\geq 1$. Since \eqref{def_of_E_finite} together with interpolation and the Calder\'on-Zygmund inequality imply that $|u|^3+|p|^{3/2} \in L^1 ((-1,0 )\times \RR^3 )$, we see that $\int_{-1}^0 \int_{B_1(x_n)} \left( |u|^3 + |p|^{3/2} \right) > \epsilon $ for only finitely many $n$'s. Thus there exists $R>0$ such that $|u| \leq C (\epsilon )$ on $(-1/2,0) \times \{ |x|>R \}$. Hence $\| u(t) \|_{L^\infty (\RR^3 )}$ remains bounded as $t\to 0^-$, and so regularity of $u$ persist beyond $t=0$, due to the classical Leray estimates (see Corollary 6.25 in \cite{OP}, for example). This concludes the proof of Theorem~\ref{thm_main} once we establish the next lemma.
\end{proof}
Let us introduce some notation. Given a suitable weak solution $(u,\pi )$ we denote by 
\[
\Q_r \coloneqq B_r \times (-r^2,0) 
\]
the finite cylinder of radius $r$ and set
\[\begin{split}
P(\pi ,r) &\coloneqq \frac{1}{r^2} \int_{\Q_r } |\pi |^{\frac32},\hspace{2.7cm} 
C(u,r) \coloneqq \frac{1}{r^2} \int_{\Q_r } |u|^{3},\\
A(u,r) &\coloneqq \frac{1}{r} \sup_{t\in(-r^2 , 0]} \int_{B_r } |u(t)|^2 \d x, \qquad E(u,r) \coloneqq \frac1{r} \int_{\Q_r} |\na u |^2 .
\end{split}
\]
We may now state the lemma that we used above.
\begin{lemma}[Conditional local regularity]\label{lem_to_show}
Given $M>0$ there exists $\varepsilon (M)>0$ with the following property: If $(u,\pi)$ is a suitable weak solution in $\Q_1$ such that  
\[
\sup_{r\in (0,1)}\frac1{r}\left( \sup_{t\in(-r^2 , 0]} \int_{B_r } |u(t)|^2 \d x+\int_{\Q_r} |\na u |^2 \right)   \leq M <\infty
\]
then $(0,0)$ is a regular point provided
\[
r_0^{2-\frac3{b}-\frac2{a} } \| \p_3 u \|_{L^{a} ((-r_0^2 ,0); L^{b} (B_{r_0}))} \leq \varepsilon (M)
\]
for some $r_0\in (0,C_0(u,\pi))$ and $a,b\geq 1$. 
\end{lemma}
\begin{proof}[Proof of Lemma~\ref{lem_to_show}.]
Note that, by interpolation (cf.~\cite[Lemma~2.1]{ozanskiCKN}, for example), the first assumption gives
\[
C(u,r) \lec M
\]
for every $r\in (0,1)$. We shall show the claim with
\[
C_0 (u,\pi )\coloneqq \min \left\{ \frac12 , \bigl(C(u,1)+P(\pi ,1)\bigr)^{-2} \right\}.
\]
Suppose that the claim does not hold. Then there exists a sequence $(u^k, \pi^k )$ and $r_k\in (0, C_0 (u^k,\pi^k))$ such that
\[
C(u^k,r)\lec M\quad \text{ and }\quad
 r_k^{2-\frac3{b}-\frac2{a} } \| \p_3 u^k \|_{L^{a}_t L^{b}_x  (\Q_{r_k} ) } \leq \frac1{k}
\]
while $(0,0)$ is a singular point of $u^k$ for every $k$. Using \cite[Lemma~A.2]{wang_zhang_2014}, we obtain
\[
A(u^k , r) + E(u^k ,r) + P(\pi^k ,r) \leq C(M)
,
\]
for all $r\in(0,r_k)$. In order to relax the restriction on the range of $r$ we apply the rescaling
\[
v^k (x,t) \coloneqq r_k u^k (r_k x , r_k^2 t)
\comma
 q^k (x,t) \coloneqq r_k^2 \pi^k (r_k x , r_k^2 t)
\]
to obtain
\eqnb
\notag
A(v^k , r) + E(v^k ,r) + P(q^k ,r) + C(v^k ,r) + k \| \p_3 v^k \|_{L^a_t L^b_x (\Q_1 )} \leq C(M)
\eqne
for all $r\in(0,1)$. This estimate on $(v^k,q^k)$ together with the Aubin-Lions Lemma (cf.~\cite[Theorem~4.12]{NSE_book}, for example) is sufficient to extract a subsequence, which we relabel, such that
\[
v^k \to v \text{ in } L^3 (\Q_{1/2}), \quad q^k \rightharpoonup q \text{ in } L^{\frac32} (\Q_{1/2}), \quad \text{ and } \quad \p_3 v^k \to 0 \text{ in } L^a_t L^b_x (\Q_1 ),
\]
where $(v,q)$ is a suitable weak solution to the Navier-Stokes equations on $Q_{1/2}$ such that $\p_3 v=0$. 
 It follows that 
$v$ and $\nabla v$ are bounded functions in $\Q_{1/4}$ due to the localized regularity condition on $\p_3 v$ of \cite{KRZ_localized}.
Since also
$\int_{\Q_{1/2}}|q|^{3/2}<\infty$, we get, using the elliptic regularity
on the equation $-\Delta q=\partial_{i}v_j\partial_{j}v_i\in L^\infty(\Q_{1/4})$ at almost every time $t\in(-1/4,0]$, that
$\Vert q\Vert_{W^{2,p}(B_{1/4})}\lec \Vert \nabla v\Vert_{L^{2p}(B_{1/4})}^2+\Vert q\Vert_{L^{3/2}(B_{1/4})}$ 
for every $p\in(1,\infty)$ and
almost every $t\in(-1/4,0]$ \cite[Theorem~9.11]{GT}. Using this statement with $p$ sufficiently large, we obtain
$ q\in {L_t^{3/2}L_{x}^{\infty}(\Q_{1/8})}$.
This immediately implies
$r^{-3}\int_{\Q_r}|q|^{3/2}<\infty$ for $r\in(0,1/8)$,
and consequently for every $r\in (0,1/16)$ we have
\[\begin{split} 
\epsilon &\leq \liminf_{k\to \infty} \frac{1}{r^2} \int_{\Q_r } \left( |v^k|^3 + |q^k|^{\frac32} \right) 
= 
\frac{1}{r^2} \int_{\Q_r } \left( |v|^3 + |q|^{\frac32} \right) 
\leq
C_{v,q} (r^{3}+r) \leq C_{v,q} r
,
\end{split}
\]
where $\epsilon>0$ is given by the Caffarelli-Kohn-Nirenberg condition (cf.~\cite[Theorem~2.2]{ozanskiCKN})
and $C$ depends on $M$ and 
the uniform bound of $v$ and $\nabla v$ on $\Q_{1/4}$.
The above inequality leads to a contradiction when we send $r\to0$, concluding the proof.
\end{proof}

\colb

\end{document}